\documentclass[a4paper]{article}
\usepackage{amssymb}
\usepackage{amsmath}
\usepackage{amsfonts}
\usepackage{fontsmpl}
\usepackage{accents}
\usepackage{float}

\newtheorem{theorem}{Theorem}

\newtheorem{corollary}[theorem]{Corollary}

\newtheorem{definition}[theorem]{Definition}
\newtheorem{example}[theorem]{Example}

\newtheorem{lemma}[theorem]{Lemma}

\newenvironment{proof}[1][Proof]{\noindent\textbf{#1.} }{\ \rule{0.5em}{0.5em}}
\input{tcilatex}
\begin{document}

\title{\textbf{Nilpotent Evolution Algebras over arbitrary fields}}
\author{A. S. Hegazi, Hani Abdelwahab \\
Mathematics department, Faculty of Science, Mansoura University, Egypt\\
hegazi@mans.edu.eg\\
haniamar1985@gmail.com}
\date{}
\maketitle

\begin{abstract}
The paper is devoted to the study of annihilator extensions of evolution
algebras and suggests an approach to classify finite-dimensional nilpotent
evolution algebras. Subsequently nilpotent evolution algebras of dimension
up to four are classified.
\end{abstract}


\bigskip


\bigskip

\section{Introduction}

Evolution algebras were introduced in 2006 by J.P. Tian and P. Vojtechovsky
in their paper \textquotedblleft Mathematical concepts of evolution algebras
in non-Mendelian genetics " (see \cite{V}). Later on, Tian laid the
foundations of evolution algebras in his monograph \cite{T}. These algebras
present many connections with other mathematical fields including graph
theory, group theory, Markov chains, dynamical systems, knot theory, $3$%
-manifolds and the study of the Riemann-Zeta function (see \cite{T}).

Evolution algebras are in general non-associative and do not belong to any
of the well-known classes of non-associative algebras such as Lie algebras,
alternative algebras, or Jordan algebras. Therefore, the research on these
algebras follows different paths (see \cite{L.M}, \cite{Omriov}, \cite{J.M}, 
\cite{U.A}, \cite{J.P}).

One of the classical problems in study of any class of algebras is to know
how many different (up to isomorphism) algebras exist for each dimension. In
this way in \cite{Jordan}, \cite{A. S. Hegazi}, \cite{Mazzola80} and \cite%
{Seeley}, the classifications of nilpotent Jordan algebras, nilpotent Lie
superalgebras, nilpotent associative algebras and nilpotent Lie algebras of
low dimensions were given.

In this paper we study the class of nilpotent evolution algebras. Our aim is
to describe a method for classifying nilpotent evolution algebras. In \cite%
{Omriov}, the equivalence between nilpotent evolution algebras and evolution
algebras which are defined by upper triangular matrices is proved. In \cite%
{T}, J. Tian defined an evolution algebra associated to any directed graph.
In \cite{Elduque}, A. Elduque and A. Labra\ considered the reverse
direction, a directed graph is attached to any evolution algebra and they
proved that nilpotency of an evolution algebra is equivalent to the
nonexistence of oriented cycles in the attached directed graph.

The paper is organized as follows. In Section \ref{prelim}, we give some
basic concepts about evolution algebras. Section \ref{construction}\ is
devoted to the description of construction of evolution algebras with
non-trivial annihilator as annihilator extensions of evolution algebras of
lower dimensions. In Section \ref{method}, we describe a method for
classifying nilpotent evolution algebras. In Section \ref{dim3}, we classify
nilpotent evolution algebras of dimension up to three over arbitrary fields.
In Section \ref{dim4}, four-dimensional nilpotent evolution algebras are
classified over an algebraic closed field of any characteristic and over $%
\mathbb{R}
$.

\section{Preliminaries}

\label{prelim}

\begin{definition}
\cite{T} An \emph{evolution algebra} is an algebra ${\mathcal{E}}$
containing a basis (as a vector space) $B=\left\{ e_{1},\ldots
,e_{n}\right\} $ such that $e_{i}e_{j}=0$ for any $1\leq i<j\leq n$. A basis
with this property is called a \emph{natural basis}.
\end{definition}

Given a natural basis $B=\left\{ e_{1},\ldots ,e_{n}\right\} $ of an
evolution algebra ${\mathcal{E}}$, 
\begin{equation*}
e_{i}^{2}=\sum_{j=1}^{n}\alpha _{ij}e_{j}
\end{equation*}%
for some scalars $\alpha _{ij}\in \mathbb{F}$, $1\leq i,j\leq n$. The matrix 
$A=\bigl(\alpha _{ij}\bigr)$ is the \emph{matrix of structural constants} of
the evolution algebra ${\mathcal{E}}$, relative to the natural basis $B$.

\begin{definition}
An \emph{ideal} $\mathcal{I}$ of an evolution algebra ${\mathcal{E}}$ is an
evolution algebra satisfying ${\mathcal{E}}\mathcal{I}\subseteq \mathcal{I}$.
\end{definition}

\smallskip Given an evolution algebra ${\mathcal{E}}$, consider its \emph{%
annihilator }%
\begin{equation*}
ann\left( {\mathcal{E}}\right) :=\left\{ x\in {\mathcal{E}}:x{\mathcal{E}}%
=0\right\} .
\end{equation*}

\begin{lemma}
\cite[Lemma 2.7]{Elduque} Let $B=\left\{ e_{1},\ldots ,e_{n}\right\} $ be a
natural basis of an evolution algebra ${\mathcal{E}}$. Then 
\begin{equation*}
ann\left( {\mathcal{E}}\right) =\mathrm{span}\left\{
e_{i}:e_{i}^{2}=0\right\} .
\end{equation*}
\end{lemma}

Hence $ann\left( {\mathcal{E}}\right) $ is an ideal of the evolution algebra 
${\mathcal{E}}$. Let $B^{\prime }=\left\{ e_{r+1},\ldots ,e_{n}\right\} $ be
a natural basis of $ann\left( {\mathcal{E}}\right) $. Then the quotient ${%
\mathcal{E}}/ann\left( {\mathcal{E}}\right) $ is an evolution algebra with a
natural basis $B^{\prime \prime }=\left\{ \bar{e}_{i}=e_{i}+ann\left( {%
\mathcal{E}}\right) :1\leq i\leq r\right\} $.

Note that we can reorder the natural basis $B$ of an evolution algebra ${%
\mathcal{E}}$ so that $B^{\prime }=\{e_{r+1},\ldots ,e_{n}\}$ be a natural
basis of $ann\left( {\mathcal{E}}\right) .$

\bigskip

Given an evolution algebra, we introduce the following sequence of subspaces:%
\begin{equation*}
\begin{array}{cc}
{\mathcal{E}}^{\left\langle 1\right\rangle }={\mathcal{E}}, & {\mathcal{E}}%
^{\left\langle k+1\right\rangle }={\mathcal{E}}^{\left\langle k\right\rangle
}{\mathcal{E}}.%
\end{array}%
\end{equation*}

\begin{definition}
An algebra ${\mathcal{E}}$ is called \emph{nilpotent} if there exists $n\in {%
\mathbb{N}}$ such that ${\mathcal{E}}^{\left\langle n\right\rangle }=0$, and
the minimal such number is called the \emph{index of nilpotency}.
\end{definition}

One can easily verify that if ${\mathcal{E}}$ is nilpotent\emph{\ }then $%
ann\left( {\mathcal{E}}\right) \neq 0$. Moreover, ${\mathcal{E}}$ is
nilpotent\emph{\ }if and only if ${\mathcal{E}}/ann\left( {\mathcal{E}}%
\right) $ is nilpotent.

\section{Constructing Nilpotent evolution Algebra}

\label{construction}Let ${\mathcal{E}}$ be an evolution algebra with a
natural basis $B_{{\mathcal{E}}}=\left\{ e_{1},\ldots ,e_{m}\right\} $ and $%
V $ be a vector space with a basis $B_{V}=\left\{ e_{m+1},\ldots
,e_{n}\right\} $. Let $\theta :{\mathcal{E\times E\longrightarrow }}V$ be a
bilinear map on ${\mathcal{E}}$, set ${\mathcal{E}}_{\theta }={\mathcal{%
E\oplus }}V$ define a multiplication on ${\mathcal{E}}_{\theta }$ by:%
\begin{equation*}
e_{i}\star e_{j}=e_{i}e_{j}\mid _{{\mathcal{E}}}+\theta \left(
e_{i},e_{j}\right) \mbox{ if }1\leq i,j\leq m,\mbox{ and otherwise }%
e_{i}\star e_{j}=0.
\end{equation*}%
Then ${\mathcal{E}}_{\theta }$ is an evolution algebra with a natural basis $%
B_{{\mathcal{E}}_{\theta }}=\left\{ e_{1},\ldots ,e_{n}\right\} $ if and
only if $\theta \left( e_{i},e_{j}\right) =0$ for all $i\neq j$. Denote the
space of all bilinear $\theta $ such that $\theta \left( e_{i},e_{j}\right)
=0$ for all $i\neq j\mathcal{\ }$by ${\mathcal{Z}}\left( {\mathcal{E\times E}%
},V\right) $. For $\theta \in {\mathcal{Z}}\left( {\mathcal{E\times E}}%
,V\right) $ with $\dim V=n-m$, the evolution algebra ${\mathcal{E}}_{\theta
}={\mathcal{E\oplus }}V$ is called a $\left( n-m\right) $\emph{-dimensional
annihilator extension} of ${\mathcal{E}}$ by $V.$

\bigskip

Now, for a linear map $f\in Hom\left( {\mathcal{E}},V\right) $, if we define 
$\theta _{f}:{\mathcal{E\times E\longrightarrow }}V$ by $\theta _{f}\left(
e_{i},e_{j}\right) =f\left( e_{i}e_{j}\right) $, then $\theta _{f}\in {%
\mathcal{Z}}\left( {\mathcal{E\times E}},V\right) $. One can easily verify
that the space defined by ${\mathcal{B}}\left( {\mathcal{E\times E}}%
,V\right) =\{\theta _{f}:f\in Hom\left( {\mathcal{E}},V\right) \}$ is a
subspace of ${\mathcal{Z}}\left( {\mathcal{E\times E}},V\right) $.

\begin{lemma}
\label{equivalent cocycles}Let ${\mathcal{E}}$ be an evolution algebra with
a natural basis $B_{{\mathcal{E}}}=\left\{ e_{1},\ldots ,e_{m}\right\} $ and 
$V$ be a vector space with a basis $B_{V}=\left\{ e_{m+1},\ldots
,e_{n}\right\} $. Suppose that $\theta \in {\mathcal{Z}}\left( {\mathcal{%
E\times E}},V\right) $ and $\theta _{f}\in {\mathcal{B}}\left( {\mathcal{%
E\times E}},V\right) $. Then ${\mathcal{E}}_{\theta }\cong {\mathcal{E}}%
_{\theta +\theta _{f}}.$
\end{lemma}

\begin{proof}
Let $x=\overset{n}{\underset{i=1}{\sum }}\alpha _{i}e_{i}\in {\mathcal{E}}%
_{\theta }$, define a linear map $\sigma :{\mathcal{E}}_{\theta
}\longrightarrow {\mathcal{E}}_{\theta +\theta _{f}}$ by 
\begin{equation*}
\sigma \left( \overset{n}{\underset{i=1}{\sum }}\alpha _{i}e_{i}\right) =%
\overset{n}{\underset{i=1}{\sum }}\alpha _{i}e_{i}+f\left( \overset{m}{%
\underset{i=1}{\sum }}\alpha _{i}e_{i}\right) .
\end{equation*}%
Then $\sigma $ is an invertible linear transformation. Moreover, let $x=%
\overset{n}{\underset{i=1}{\sum }}\alpha _{i}e_{i},y=\overset{n}{\underset{%
i=1}{\sum }}\beta _{i}e_{i}\in {\mathcal{E}}_{\theta }$ then 
\begin{eqnarray*}
\sigma \left( x\star _{{\mathcal{E}}_{\theta }}y\right) &=&\sigma \left( 
\overset{m}{\underset{i=1}{\sum }}\alpha _{i}\beta _{i}\left(
e_{i}^{2}+\theta \left( e_{i},e_{i}\right) \right) \right) =\overset{m}{%
\underset{i=1}{\sum }}\alpha _{i}\beta _{i}\left( e_{i}^{2}+\theta \left(
e_{i},e_{i}\right) +f\left( e_{i}^{2}\right) \right) \\
&=&\overset{m}{\underset{i=1}{\sum }}\alpha _{i}\beta _{i}\left(
e_{i}^{2}+\left( \theta +\theta _{f}\right) \left( e_{i},e_{i}\right)
\right) =\sigma \left( x\right) \star _{{\mathcal{E}}_{\theta +\theta
_{f}}}\sigma \left( y\right) .
\end{eqnarray*}%
Hence ${\mathcal{E}}_{\theta }\cong {\mathcal{E}}_{\theta +\theta _{f}}.$
\end{proof}

Hence the isomorphism type of ${\mathcal{E}}_{\theta }$ only depends on the
element $\theta +{\mathcal{B}}\left( {\mathcal{E\times E}},V\right) $.
Therefore we consider the quotient space 
\begin{equation*}
{\mathcal{H}}\left( {\mathcal{E\times E}},V\right) ={\mathcal{Z}}\left( {%
\mathcal{E\times E}},V\right) /{\mathcal{B}}\left( {\mathcal{E\times E}}%
,V\right) .
\end{equation*}%
The annihilator extension corresponding to $\theta =0$, and consequently
corresponding to any element in ${\mathcal{B}}\left( {\mathcal{E\times E}}%
,V\right) $ is called a \emph{trivial extension}.

Suppose that ${\mathcal{E}}$ is an evolution algebra with a natural basis $%
B_{{\mathcal{E}}}=\left\{ e_{1},\ldots ,e_{m}\right\} $ and $V$ be a vector
space with a basis $B_{V}=\left\{ e_{m+1},\ldots ,e_{n}\right\} $. Then a
symmetric bilinear map $\theta \in {\mathcal{Z}}\left( {\mathcal{E\times E}}%
,V\right) $ can be written as $\theta \left( e_{i},e_{i}\right) =\underset{%
j=1}{\overset{n-m}{\sum }}\theta _{j}\left( e_{i},e_{i}\right) e_{m+j}$,
where $\theta _{j}\in {\mathcal{Z}}\left( {\mathcal{E\times E}},\mathbb{F}%
\right) $. Further, $\theta \in {\mathcal{B}}\left( {\mathcal{E\times E}}%
,V\right) $ if and only if all $\theta _{j}\in {\mathcal{B}}\left( {\mathcal{%
E\times E}},\mathbb{F}\right) $.

\begin{lemma}
\label{cobound}$\dim {\mathcal{H}}\left( {\mathcal{E\times E}},\mathbb{F}%
\right) =\dim {\mathcal{E}}-\dim {\mathcal{E}}^{\left\langle 2\right\rangle
} $.
\end{lemma}

\begin{proof}
We claim that $\dim {\mathcal{B}}\left( {\mathcal{E\times E}},\mathbb{F}%
\right) =\dim {\mathcal{E}}^{\left\langle 2\right\rangle }$. Let $%
x_{1},x_{2},\ldots ,x_{k}$ be a basis of ${\mathcal{E}}^{\left\langle
2\right\rangle }$, then $x_{1}^{\ast },x_{2}^{\ast },\ldots ,x_{k}^{\ast }$
are linearly independent in ${\mathcal{E}}^{\ast }$ where $x_{i}^{\ast }$ is
defined by $x_{i}^{\ast }(x_{i})=1$ and $x_{i}^{\ast }(x_{j})=0$ if $i\neq j$%
. Let $\theta _{f}\in {\mathcal{B}}\left( {\mathcal{E\times E}},\mathbb{F}%
\right) $, then $\theta _{f}=\alpha _{1}\theta _{x_{1}^{\ast }}+\alpha
_{2}\theta _{x_{1}^{\ast }}+\cdots +\alpha _{k}\theta _{x_{k}^{\ast }}$ and $%
\theta _{x_{1}^{\ast }},\theta _{x_{1}^{\ast }},\ldots ,\theta _{x_{k}^{\ast
}}$are linearly independent in ${\mathcal{B}}\left( {\mathcal{E\times E}},%
\mathbb{F}\right) $. Hence $\dim {\mathcal{B}}\left( {\mathcal{E\times E}},%
\mathbb{F}\right) =\dim {\mathcal{E}}^{\left\langle 2\right\rangle }$ and $%
\dim {\mathcal{H}}\left( {\mathcal{E\times E}},\mathbb{F}\right) =\dim {%
\mathcal{E}}-\dim {\mathcal{E}}^{\left\langle 2\right\rangle }$.
\end{proof}

\bigskip

Suppose that ${\mathcal{E}}$ is an evolution algebra with a natural basis $%
B_{{\mathcal{E}}}=\left\{ e_{1},\ldots ,e_{m}\right\} $. Let us denote by $%
Sym\left( {\mathcal{E}}\right) $ the space of all symmetric bilinear forms
on ${\mathcal{E}}$. Then $Sym\left( {\mathcal{E}}\right) =\left\langle
\delta _{e_{i},e_{j}}:1\leq i\leq j\leq m\right\rangle $ while ${\mathcal{Z}}%
\left( {\mathcal{E\times E}},\mathbb{F}\right) =\left\langle \delta
_{e_{i},e_{i}}:1=1,2,\ldots ,m\right\rangle $, where a bilinear form $\delta
_{e_{i},e_{j}}:{\mathcal{E\times E\longrightarrow }}\mathbb{F}$ is defined
by $\delta _{e_{i},e_{j}}\left( e_{l},e_{m}\right) =1$ if $\left\{
i,j\right\} =\left\{ l,m\right\} $, and otherwise it takes the value zero.

\begin{example}
Let ${\mathcal{E}}:e_{1}^{2}=e_{2}$ be a 2-dimensional nilpotent evolution
algebra. Here ${\mathcal{Z}}\left( {\mathcal{E\times E}},\mathbb{F}\right)
=\left\langle \delta _{e_{1},e_{1}},\delta _{e_{2},e_{2}}\right\rangle $, by 
\emph{Lemma}$\ $\emph{\ref{cobound}}, ${\mathcal{B}}\left( {\mathcal{E\times
E}},\mathbb{F}\right) $ is spanned by $\theta _{e_{2}^{\ast }}$. Let $\theta
_{e_{2}^{\ast }}=\alpha \delta _{e_{1},e_{1}}+\beta \delta _{e_{2},e_{2}}\in 
{\mathcal{Z}}\left( {\mathcal{E\times E}},\mathbb{F}\right) $. It follows
that $\alpha =\theta _{e_{2}^{\ast }}\left( e_{1},e_{1}\right) =e_{2}^{\ast
}\left( e_{1}^{2}\right) =1$, and $\beta =\theta _{e_{2}^{\ast }}\left(
e_{2},e_{2}\right) =e_{2}^{\ast }\left( e_{2}^{2}\right) =0$ then $\theta
_{e_{2}^{\ast }}=\delta _{e_{1},e_{1}}$. Hence ${\mathcal{B}}\left( {%
\mathcal{E\times E}},\mathbb{F}\right) =\left\langle \delta
_{e_{1},e_{1}}\right\rangle $ and ${\mathcal{H}}\left( {\mathcal{E\times E}},%
\mathbb{F}\right) =\left\langle \delta _{e_{2},e_{2}}\right\rangle $.
\end{example}

\begin{example}
Let ${\mathcal{E}}:e_{1}^{2}=e_{3},e_{2}^{2}=e_{3}$ be a 3-dimensional
nilpotent evolution algebra. Here ${\mathcal{Z}}\left( {\mathcal{E\times E}},%
\mathbb{F}\right) =\left\langle \delta _{e_{1},e_{1}},\delta
_{e_{2},e_{2}},\delta _{e_{3},e_{3}}\right\rangle $, by \emph{Lemma} \emph{%
\ref{cobound}}, ${\mathcal{B}}\left( {\mathcal{E\times E}},\mathbb{F}\right) 
$ is spanned by $\theta _{e_{3}^{\ast }}$. Let $\theta _{e_{3}^{\ast }}=%
\underset{i=1}{\overset{3}{\sum }}\alpha _{i}\delta _{e_{i},e_{i}}$. Then $%
\alpha _{1}=\alpha _{2}=1$, and $\alpha _{3}=0$. Hence ${\mathcal{B}}\left( {%
\mathcal{E\times E}},\mathbb{F}\right) =\left\langle \delta
_{e_{1},e_{1}}+\delta _{e_{2},e_{2}}\right\rangle $ and ${\mathcal{H}}\left( 
{\mathcal{E\times E}},\mathbb{F}\right) =\left\langle \delta
_{e_{1},e_{1}},\delta _{e_{3},e_{3}}\right\rangle $.
\end{example}

\bigskip

Suppose that ${\mathcal{E}}_{\theta }={\mathcal{E\oplus }}V$ is an evolution
algebra then $VV=0$ and so $ann({\mathcal{E}}_{\theta })\supseteq V\neq 0$.
It means that by this way we construct evolution algebras with $ann({%
\mathcal{E}}_{\theta })\neq 0$. Further, any evolution algebras with $%
ann\left( {\mathcal{E}}\right) \neq 0$ can be obtained by this way.

\begin{lemma}
\label{construct}Let ${\mathcal{E}}$ be an evolution algebra and $ann\left( {%
\mathcal{E}}\right) \neq 0$. Then ${\mathcal{E}}$ is an annihilator
extension of a lower dimensional evolution algebra.
\end{lemma}

\begin{proof}
Suppose that ${\mathcal{E}}$ is an evolution algebra with a natural basis $%
B=\left\{ e_{1},\ldots ,e_{n}\right\} $. Let $B^{\prime }=\left\{
e_{r+1},\ldots ,e_{n}\right\} $ be a natural basis of $ann\left( {\mathcal{E}%
}\right) $ then the quotient ${\mathcal{E}}/ann\left( {\mathcal{E}}\right) $
is an evolution algebra with a natural basis $B^{\prime \prime }=\{\bar{e}%
_{i}=e_{i}+ann\left( {\mathcal{E}}\right) :1\leq i\leq r\}$. If $\bar{x}=%
\overset{r}{\underset{i=1}{\sum }}\alpha _{i}\bar{e}_{i}\in {\mathcal{E}}%
/ann\left( {\mathcal{E}}\right) $\ then we define an injective linear map $%
\sigma :{\mathcal{E}}/ann\left( {\mathcal{E}}\right) \longrightarrow {%
\mathcal{E}}$ by $\sigma \left( \bar{x}\right) =\overset{r}{\underset{i=1}{%
\sum }}\alpha _{i}e_{i}$. For any $\bar{x},\bar{y}\in {\mathcal{E}}%
/ann\left( {\mathcal{E}}\right) $ we have $\sigma (\bar{x})\sigma (\bar{y}%
)-\sigma (\bar{x}\bar{y})\in ann\left( {\mathcal{E}}\right) $ (let $\pi :{%
\mathcal{E}}\longrightarrow {\mathcal{E}}/ann\left( {\mathcal{E}}\right) $
be a projection map then $\pi \left( \sigma \left( \bar{x}\right) \right) =%
\bar{x}$ and hence $\sigma (\bar{x})\sigma (\bar{y})-\sigma (\bar{x}\bar{y}%
)\in \ker \pi $). Then $\theta :{\mathcal{E}}/ann\left( {\mathcal{E}}\right)
\times {\mathcal{E}}/ann\left( {\mathcal{E}}\right) \longrightarrow
ann\left( {\mathcal{E}}\right) $ defined by $\theta (\bar{x},\bar{y})=\sigma
(\bar{x})\sigma (\bar{y})-\sigma (\bar{x}\bar{y})\in {\mathcal{Z}}\left( {%
\mathcal{E\times E}},V\right) $. It remains to show that $({\mathcal{E}}%
/ann\left( {\mathcal{E}}\right) )_{\theta }={\mathcal{E}}/ann\left( {%
\mathcal{E}}\right) \oplus ann\left( {\mathcal{E}}\right) $ is isomorphic to 
${\mathcal{E}}$. If $x\in {\mathcal{E}}$, then $x$ can be uniquely written
as $x=\sigma (\bar{y})+z$, where $\bar{y}\in {\mathcal{E}}/ann\left( {%
\mathcal{E}}\right) $ and $z\in ann\left( {\mathcal{E}}\right) $. Defining $%
\phi :{\mathcal{E}}\longrightarrow ({\mathcal{E}}/ann\left( {\mathcal{E}}%
\right) )_{\theta }$ by $\phi (x)=\phi \left( \sigma (\bar{y})+z\right) =%
\bar{y}+z$, then $\phi $ is bijective and%
\begin{eqnarray*}
\phi (x_{1}x_{2}) &=&\phi \left( (\sigma (\bar{y}_{1})+z_{1})(\sigma (\bar{y}%
_{2})+z_{2})\right) =\phi (\sigma (\bar{y}_{1})\sigma (\bar{y}_{2})) \\
&=&\phi (\sigma (\bar{y}_{1}\bar{y}_{2})+\theta (\bar{y}_{1},\bar{y}_{2}))=%
\bar{y}_{1}\bar{y}_{2}+\theta (\bar{y}_{1},\bar{y}_{2}) \\
&=&(\bar{y}_{1}+z_{1})\star (\bar{y}_{2}+z_{2})=\phi (x_{1})\star \phi
(x_{2}).
\end{eqnarray*}%
Then $\phi $ is an isomorphism.
\end{proof}

So in particular all nilpotent evolution algebras can be obtained by this
way.

\section{The Classification Method}

\label{method}In Section \ref{construction}, Lemma \ref{construct}, we show
that a nilpotent evolution algebra ${\mathcal{E}}$ can be obtained as an
annihilator extension of $M={\mathcal{E}}/ann\left( {\mathcal{E}}\right) $
by $V=ann\left( {\mathcal{E}}\right) $. Now, suppose that ${\mathcal{E}}%
_{\theta }$ is an annihilator extension of ${\mathcal{E}}$ by $V$, if $ann({%
\mathcal{E}}_{\theta })\neq V$ then ${\mathcal{E}}_{\theta }$ can be
obtained as an annihilator extension of ${\mathcal{E}}/ann\left( {\mathcal{E}%
}\right) $ by $ann\left( {\mathcal{E}}\right) $. So we get ${\mathcal{E}}%
_{\theta }$ as annihilator extensions of different algebras ${\mathcal{E}}$
and ${\mathcal{E}}/ann\left( {\mathcal{E}}\right) $. Our aim is to avoid
construct the same algebra by different algebras.

\begin{definition}
Let $\theta \in {\mathcal{Z}}\left( {\mathcal{E\times E}},V\right) $ then $%
\theta ^{\bot }=\left\{ e_{i}\in {\mathcal{E}}:\theta
(e_{i},e_{i})=0\right\} $ is called the \emph{radical} of $\theta $. If $%
\theta (e_{i},e_{i})=\underset{j=1}{\overset{n-m}{\sum }}\theta _{j}\left(
e_{i},e_{i}\right) e_{m+j}$ where $\theta _{j}\in {\mathcal{Z}}\left( {%
\mathcal{E\times E}},\mathbb{F}\right) $, then 
\begin{equation*}
\theta ^{\bot }=\theta _{1}^{\bot }\cap \theta _{2}^{\bot }\cap \cdots \cap
\theta _{n-m}^{\bot }.
\end{equation*}
\end{definition}

\begin{lemma}
\label{rad}Let $\theta \in {\mathcal{Z}}\left( {\mathcal{E\times E}}%
,V\right) $. Then $ann({\mathcal{E}}_{\theta })=\left( \theta ^{\bot }\cap
ann\left( {\mathcal{E}}\right) \right) \oplus V.$
\end{lemma}

\begin{proof}
If $e_{i}\in ann({\mathcal{E}}_{\theta })$ then $e_{i}\star
e_{i}=e_{i}^{2}+\theta \left( e_{i},e_{i}\right) =0$,$\ $hence $%
e_{i}^{2}=\theta \left( e_{i},e_{i}\right) =0$. It follows that $e_{i}\in
\left( \theta ^{\bot }\cap ann\left( {\mathcal{E}}\right) \right) \oplus V$.
On the other hand, note that $V\subset ann({\mathcal{E}}_{\theta })$. If $%
e_{i}\in \theta ^{\bot }\cap ann\left( {\mathcal{E}}\right) \ $then $%
e_{i}\star e_{i}=e_{i}^{2}+\theta \left( e_{i},e_{i}\right) =0$ and hence $%
e_{i}\in ann({\mathcal{E}}_{\theta })$.
\end{proof}

\begin{corollary}
Let $\theta \in {\mathcal{Z}}\left( {\mathcal{E\times E}},V\right) $. Then $%
ann({\mathcal{E}}_{\theta })=V$ if and only if 
\begin{equation*}
\theta ^{\bot }\cap ann\left( {\mathcal{E}}\right) =0.
\end{equation*}
\end{corollary}

\begin{corollary}
\label{diff}Let$\mathcal{\ }\theta \in {\mathcal{Z}}\left( {\mathcal{E\times
E}},V\right) $ such that $\theta ^{\bot }\cap ann\left( {\mathcal{E}}\right)
=0$. Then ${\mathcal{E}}_{\theta }$ can not be obtained as an annihilator
extension of a different algebra.
\end{corollary}

\begin{proof}
Consider $\vartheta \in {\mathcal{Z}}({\mathcal{M\times M}},W)$ such that $%
\vartheta ^{\bot }\cap ann({\mathcal{M}})=0$. If ${\mathcal{M}}_{\vartheta
}\cong {\mathcal{E}}_{\theta }$, then $ann\left( {\mathcal{M}}_{\vartheta
}\right) \cong ann\left( {\mathcal{E}}_{\theta }\right) $. Hence ${\mathcal{M%
}}\cong {\mathcal{M}}_{\vartheta }/ann\left( {\mathcal{M}}_{\vartheta
}\right) \cong {\mathcal{E}}_{\theta }/ann\left( {\mathcal{E}}_{\theta
}\right) \cong {\mathcal{E}}$.
\end{proof}

Thus to avoid constructing the same algebra as an annihilator extension of
different algebras, we restrict $\theta $ to be such that $\theta ^{\bot
}\cap ann\left( {\mathcal{E}}\right) =0$.

\bigskip

Suppose that ${\mathcal{E}}$ is an evolution algebra. Denote $Aut\left( {%
\mathcal{E}}\right) $ for the set of all automorphisms of ${\mathcal{E}}$.
Let $\phi \in Aut\left( {\mathcal{E}}\right) $. For $\theta \in {\mathcal{Z}}%
\left( {\mathcal{E\times E}},V\right) $ defining $\phi \theta \left(
e_{i},e_{i}\right) =\theta \left( \phi \left( e_{i}\right) ,\phi \left(
e_{i}\right) \right) $ then $\phi \theta \in {\mathcal{Z}}\left( {\mathcal{%
E\times E}},V\right) $ if and only if $\theta \left( \phi \left(
e_{i}\right) ,\phi \left( e_{j}\right) \right) =0$ for all $i\neq j$.

Consider the matrix representation and assume that $\theta =\overset{m}{%
\underset{i=1}{\sum }}c_{ii}\delta _{e_{i},e_{i}}\in {\mathcal{Z}}\left( {%
\mathcal{E\times E}},\mathbb{F}\right) $. Then $\theta $ can be represented
by an $m\times m$ diagonal matrix $\big(c_{ii}\big)$, and then $\phi \theta
=\phi ^{t}\big(c_{ii}\big)\phi $. Hence $\phi \theta \in {\mathcal{Z}}\left( 
{\mathcal{E\times E}},V\right) $ if and only if $\phi \theta $\ be a
diagonal matrix.

\bigskip

Now, suppose that ${\mathcal{E}}$ be an evolution algebra with a natural
basis $B_{{\mathcal{E}}}=\left\{ e_{1},\ldots ,e_{m}\right\} $ and $V$ be a
vector space with a basis $B_{V}=\left\{ e_{m+1},\ldots ,e_{n}\right\} $.
Let $\theta (e_{l},e_{l})=\underset{i=1}{\overset{n-m}{\sum }}\theta
_{i}(e_{l},e_{l})e_{m+i},\vartheta (e_{l},e_{l})=\underset{i=1}{\overset{n-m}%
{\sum }}\vartheta _{i}(e_{l},e_{l})e_{m+i}\in {\mathcal{H}}\left( {\mathcal{%
E\times E}},V\right) $ and $\theta ^{\bot }\cap ann\left( {\mathcal{E}}%
\right) =\vartheta ^{\bot }\cap ann\left( {\mathcal{E}}\right) =0$, i.e., $%
ann({\mathcal{E}}_{\theta })=ann({\mathcal{E}}_{\vartheta })=V$. Assume that 
${\mathcal{E}}_{\theta }$ and ${\mathcal{E}}_{\vartheta }$ are isomorphic,
let $\phi :{\mathcal{E}}_{\theta }\longrightarrow {\mathcal{E}}_{\vartheta }$
be an isomorphism then $\phi \left( ann({\mathcal{E}}_{\theta })\right) =ann(%
{\mathcal{E}}_{\vartheta })$, i.e., $\psi =\phi \mid _{V}\in Gl\left(
V\right) $. If $\phi \left( e_{i}\right) =y_{i}+v_{i}$ where $y_{i}\in {%
\mathcal{E}},v_{i}$ $\in V$, then $\phi $ induces an isomorphism $\phi _{0}:{%
\mathcal{E}}\longrightarrow {\mathcal{E}}$ defined by $\phi _{0}\left(
e_{i}\right) =y_{i}$ and a linear map $f:{\mathcal{E}}\longrightarrow V$
defined by $f\left( e_{i}\right) =v_{i}$. So, we can realize $\phi $ as a
matrix of the form:%
\begin{equation*}
\phi =%
\begin{bmatrix}
\phi _{0} & 0 \\ 
f & \psi%
\end{bmatrix}%
;\phi _{0}\in Aut\left( {\mathcal{E}}\right) ,f\in Hom\left( {\mathcal{E}}%
,V\right) \mbox{ and }\psi =\phi \mid _{V}\in Gl\left( V\right) .
\end{equation*}%
Now, let $x=\underset{i=1}{\overset{n}{\sum }}a_{i}e_{i}\in {\mathcal{E}}%
_{\theta }={\mathcal{E}}\oplus V$ then%
\begin{equation*}
\phi \left( \underset{i=1}{\overset{n}{\sum }}a_{i}e_{i}\right) =\phi
_{0}\left( \underset{i=1}{\overset{m}{\sum }}a_{i}e_{i}\right) +f\left( 
\underset{i=1}{\overset{m}{\sum }}a_{i}e_{i}\right) +\psi \left( \underset{%
i=m+1}{\overset{n}{\sum }}a_{i}e_{i}\right) .
\end{equation*}%
According to the definition of $\phi $ we have 
\begin{equation*}
\phi \left( \underset{i=1}{\overset{n}{\sum }}a_{i}e_{i}\star _{{\mathcal{E}}%
_{\theta }}\underset{i=1}{\overset{n}{\sum }}a_{i}e_{i}\right) =\phi \left( 
\underset{i=1}{\overset{n}{\sum }}a_{i}e_{i}\right) \star _{{\mathcal{E}}%
_{\vartheta }}\phi \left( \underset{i=1}{\overset{n}{\sum }}%
a_{i}e_{i}\right) .
\end{equation*}

\begin{eqnarray*}
L.H.S. &=&\phi \left( \underset{i=1}{\overset{m}{\sum }}a_{i}^{2}\left[
e_{i}^{2}+\theta (e_{i},e_{i})\right] \right)  \\
&=&\overset{m}{\underset{i=1}{\sum }}a_{i}^{2}\left[ \phi _{0}\left(
e_{i}^{2}\right) +f\left( e_{i}^{2}\right) +\psi \left( \theta
(e_{i},e_{i})\right) \right]  \\
&=&\overset{m}{\underset{i=1}{\sum }}a_{i}^{2}\left[ \phi _{0}\left(
e_{i}^{2}\right) +\left( \theta _{f}+\psi \circ \theta \right) (e_{i},e_{i})%
\right] .
\end{eqnarray*}%
\begin{eqnarray*}
R.H.S. &=&\left[ \phi _{0}\left( \underset{i=1}{\overset{m}{\sum }}%
a_{i}e_{i}\right) \right] ^{2}+\vartheta \left( \phi _{0}\left( \underset{i=1%
}{\overset{m}{\sum }}a_{i}e_{i}\right) ,\phi _{0}\left( \underset{i=1}{%
\overset{m}{\sum }}a_{i}e_{i}\right) \right)  \\
&=&\overset{m}{\underset{i=1}{\sum }}a_{i}^{2}\left[ \phi _{0}\left(
e_{i}^{2}\right) +\phi _{0}\vartheta (e_{i},e_{i})\right] .
\end{eqnarray*}%
This yields to: 
\begin{equation}
\phi _{0}\vartheta (e_{i},e_{i})=\left( \theta _{f}+\psi \circ \theta
\right) (e_{i},e_{i})\mbox{ for }i=1,2,...,m.  \label{auto}
\end{equation}%
i.e., $\phi _{0}\upsilon =\theta _{f}+\psi \circ \theta \in {\mathcal{Z}}%
\left( {\mathcal{E\times E}},V\right) $. Thus ${\mathcal{E}}_{\theta }$ and $%
{\mathcal{E}}_{\vartheta }$ are isomorphic if and only if there exist $\phi
_{0}\in Aut\left( {\mathcal{E}}\right) $ and $\psi \in Gl\left( V\right) $
such that $\phi _{0}\upsilon =\psi \circ \theta $\ modulo ${\mathcal{B}}%
\left( {\mathcal{E\times E}},V\right) $.

Note that for all $i\neq j\mathcal{\ }$we have 
\begin{equation*}
0=\phi \left( e_{i}\star _{{\mathcal{E}}_{\theta }}e_{j}\right) =\phi \left(
e_{i}\right) \star _{{\mathcal{E}}_{\vartheta }}\phi \left( e_{j}\right)
=\vartheta \left( \phi _{0}\left( e_{i}\right) ,\phi _{0}\left( e_{j}\right)
\right) .
\end{equation*}%
i.e., $\phi _{0}\upsilon \in {\mathcal{H}}\left( {\mathcal{E\times E}}%
,V\right) .$

Now, for $1\leq i\leq n-m$, if $\psi \left( e_{m+i}\right) =\overset{n-m}{%
\underset{j=1}{\sum }}a_{ij}e_{m+j}$, then 
\begin{equation*}
\underset{i=1}{\overset{n-m}{\sum }}\phi _{0}\vartheta
_{i}(e_{l},e_{l})e_{m+i}=\underset{i=1}{\overset{n-m}{\sum }}\theta
_{i}(e_{l},e_{l})\psi \left( e_{m+i}\right) =\overset{n-m}{\underset{j=1}{%
\sum }}\overset{n-m}{\underset{i=1}{\sum }}a_{ij}\theta
_{i}(e_{l},e_{l})e_{m+j}.
\end{equation*}%
Hence $\phi _{0}\vartheta _{j}=\underset{i=1}{\overset{n-m}{\sum }}%
a_{ij}\theta _{i}$ modulo ${\mathcal{B}}\left( {\mathcal{E\times E}}%
,V\right) $. It follows that the $\phi _{0}\vartheta _{i}$ span the same
subspace of ${\mathcal{H}}\left( {\mathcal{E\times E}},V\right) $ as the $%
\theta _{i}$. We have proved

\begin{lemma}
\label{iso}Let $\theta (e_{l},e_{l})=\underset{i=1}{\overset{n-m}{\sum }}%
\theta _{i}(e_{l},e_{l})e_{m+i}$ and $\vartheta (e_{l},e_{l})=\underset{i=1}{%
\overset{n-m}{\sum }}\vartheta _{i}(e_{l},e_{l})e_{m+i}$ be two elements of $%
{\mathcal{H}}\left( {\mathcal{E\times E}},V\right) $. Suppose that $\theta
^{\perp }\cap ann\left( {\mathcal{E}}\right) =\vartheta ^{\perp }\cap
ann\left( {\mathcal{E}}\right) =0$. Then ${\mathcal{E}}_{\theta }\cong {%
\mathcal{E}}_{\vartheta }$ if and only if there is a $\phi \in Aut\left( {%
\mathcal{E}}\right) $ such that the $\phi \vartheta _{i}$ span the same
subspace of ${\mathcal{H}}\left( {\mathcal{E\times E}},\mathbb{F}\right) $
as the $\theta _{i}$.
\end{lemma}

In case of $\theta =\vartheta $, we obtain from Equation $\left( \ref{auto}%
\right) $ the following description of $Aut({\mathcal{E}}_{\theta })$.

\begin{lemma}
Let ${\mathcal{E}}$ be a nilpotent evolution algebra and $\theta \in {%
\mathcal{H}}\left( {\mathcal{E\times E}},V\right) $ such that $\theta ^{\bot
}\cap ann\left( {\mathcal{E}}\right) =0$. Then $Aut\left( {\mathcal{E}}%
_{\theta }\right) $ consists of all linear operators of the matrix form 
\begin{equation*}
\phi =%
\begin{bmatrix}
\phi _{0} & 0 \\ 
f & \psi%
\end{bmatrix}%
;\phi _{0}\in Aut\left( {\mathcal{E}}\right) ,f\in Hom({\mathcal{E}},V)\text{
and }\psi =\phi \mid _{V}\in Gl\left( V\right) .
\end{equation*}%
such that $\phi _{0}\theta (e_{i},e_{i})=f\left( e_{i}^{2}\right) +\psi
\left( \theta (e_{i},e_{i})\right) \mathcal{\ }$for $i=1,2,...,m$.
\end{lemma}

\bigskip

Let ${\mathcal{E}}$ be an evolution algebra with a natural basis $B_{{%
\mathcal{E}}}=\left\{ e_{1},\ldots ,e_{m}\right\} $. If ${\mathcal{E}}%
=\left\langle e_{1},..,e_{r-1},e_{r+1},..,e_{m}\right\rangle \oplus \mathbb{F%
}e_{r}$ is the direct sum of two nilpotent ideals, the $1$-dimensional ideal 
$\mathbb{F}e_{r}$ is called an \emph{annihilator component} of ${\mathcal{E}}
$. Any $m$-dimensional nilpotent evolution algebra with annihilator component%
\emph{\ }is a $1$-dimensional trivial extension of $\left( m-1\right) $%
-dimensional nilpotent evolution algebra. Our aim is to avoid construct
nilpotent algebras with annihilator components, and hence when we construct $%
\left( n+1\right) $-dimensional nilpotent algebras, we have to add algebras $%
{\mathcal{E}}=I\oplus \mathbb{F}e_{n+1}$ where ${\mathcal{I}}$ be a
nilpotent algebra with a natural basis $B_{{\mathcal{I}}}=\left\{
e_{1},\ldots ,e_{n}\right\} $.

\begin{lemma}
Let $\theta (e_{l},e_{l})=\underset{i=1}{\overset{n-m}{\sum }}\theta
_{i}(e_{l},e_{l})e_{m+i}$ and $\theta ^{\bot }\cap ann\left( {\mathcal{E}}%
\right) =0$. Then ${\mathcal{E}}_{\theta }$ has an annihilator component if
and only if$\mathcal{\ }\theta _{1},\theta _{2},\ldots ,\theta _{s}$ are
linearly dependent in ${\mathcal{H}}\left( {\mathcal{E\times E}},\mathbb{F}%
\right) $.
\end{lemma}

\begin{proof}
Suppose that ${\mathcal{E}}_{\theta }$ has an annihilator component. Then we
can write ${\mathcal{E}}_{\theta }=I\oplus \mathbb{F}e_{m+r}$ where $%
e_{m+r}\in V$. Then 
\begin{equation*}
\theta :{\mathcal{E}}\times {\mathcal{E}}\longrightarrow \left\langle
e_{m+1},\ldots ,e_{m+r-1},e_{m+r+1},\ldots ,e_{n-m}\right\rangle ,
\end{equation*}%
and hence $\theta _{r}=0$. Then $\theta _{1},\theta _{2},\ldots ,\theta
_{n-m}$ are linearly dependent. On the other hand, suppose that $\theta
_{1},\theta _{2},\ldots ,\theta _{n-m}$ are linearly dependent, let $%
\vartheta (e_{l},e_{l})=\overset{s}{\underset{i=1}{\sum }}\vartheta
_{i}(e_{l},e_{l})e_{m+i}$ where $\left\langle \vartheta _{i}:i=1,2,\ldots
,s\right\rangle $ is a proper subspace of $\left\langle \theta
_{i}:i=1,2,\ldots ,n-m\right\rangle $ and span the same space as $\theta
_{i} $. Then by Lemma \ref{iso}, we have ${\mathcal{E}}_{\theta }\cong {%
\mathcal{E}}_{\vartheta }$. It follows that ${\mathcal{E}}_{\theta }$ has an
annihilator component.
\end{proof}

\bigskip

Suppose that ${\mathcal{E}}$ is an evolution algebra. For $\theta \in {%
\mathcal{H}}\left( {\mathcal{E\times E}},\mathbb{F}\right) $, define a
subgroup $\mathcal{S}_{\theta }\left( {\mathcal{E}}\right) $ of $Aut\left( {%
\mathcal{E}}\right) $ by 
\begin{equation*}
\mathcal{S}_{\theta }\left( {\mathcal{E}}\right) =\left\{ \phi \in Aut\left( 
{\mathcal{E}}\right) :\theta \left( \phi \left( e_{i}\right) ,\phi \left(
e_{j}\right) \right) =0\text{ \ }\forall i\neq j\right\} .
\end{equation*}

Now, we have a procedure takes a nilpotent evolution algebra ${\mathcal{E}}$%
\ of dimension $m$ with a natural basis $B_{{\mathcal{E}}}=\left\{
e_{1},\ldots ,e_{m}\right\} $ gives us all nilpotent evolution algebras ${%
\mathcal{\tilde{E}}}$ of dimension $n$ with a natural basis $B_{{\mathcal{%
\tilde{E}}}}=\left\{ e_{1},\ldots ,e_{m},e_{m+1},\ldots ,e_{n}\right\} $
such that ${\mathcal{\tilde{E}}}/ann\left( {\mathcal{\tilde{E}}}\right)
\cong {\mathcal{E}}$ and ${\mathcal{\tilde{E}}}$ has no annihilator
components. The procedure runs as follows:

\begin{enumerate}
\item Determine ${\mathcal{Z}}\left( {\mathcal{E\times E}},\mathbb{F}\right)
,{\mathcal{B}}\left( {\mathcal{E\times E}},\mathbb{F}\right) $ and ${%
\mathcal{H}}\left( {\mathcal{E\times E}},\mathbb{F}\right) $.

\item Let $V$ be a vector space with a basis $B_{V}=\left\{ e_{m+1},\ldots
,e_{n}\right\} $. Consider $\theta \in {\mathcal{H}}\left( {\mathcal{E\times
E}},V\right) $ with 
\begin{equation*}
\theta (e_{l},e_{l})=\underset{i=1}{\overset{n-m}{\sum }}\theta
_{i}(e_{l},e_{l})e_{m+i},
\end{equation*}%
where the $\theta _{i}\in {\mathcal{H}}\left( {\mathcal{E\times E}},\mathbb{F%
}\right) $ are linearly independent, and $\theta ^{\bot }\cap ann\left( {%
\mathcal{E}}\right) =0$. Use the action of $\mathcal{S}_{\theta }\left( {%
\mathcal{E}}\right) $ on $\theta $ by hand calculations to get a list of
orbit representatives as small as possible.

\item For each orbit found, construct the corresponding algebra.
\end{enumerate}

\section{Nilpotent evolution algebras of dimension $\leq 3$}

\label{dim3}In this section we present a complete classification of
nilpotent evolution algebras of dimension $\leq 3$. We denote the $j$-th
algebra of dimension $i$ by ${\mathcal{E}}_{i,j}$ and among the natural
basis of ${\mathcal{E}}_{i,j}$ the multiplication is specified by giving
only the nonzero products.

\subsection{ Nilpotent evolution algebras of dimension 1}

Let ${\mathcal{E}}$ be an evolution algebra of dimension $1$ with natural
basis $\{e_{1}\}$. If ${\mathcal{E}}$ is nilpotent then $ann\left( {\mathcal{%
E}}\right) =\left\langle e_{1}\right\rangle $. Therefore, there is only one
nilpotent evolution algebra of dimension $1$, that is ${\mathcal{E}}_{1,1}$%
with a natural basis $\{e_{1}\}$ and%
\begin{table}[H] \centering%
$%
\begin{tabular}{|c|c|c|c|c|}
\hline
${\mathcal{E}}$ & Multiplication Table & ${\mathcal{H}}\left( {\mathcal{%
E\times E}},\mathbb{F}\right) $ & $ann\left( {\mathcal{E}}\right) $ & $%
Aut\left( {\mathcal{E}}\right) $ \\ \hline
${\mathcal{E}}_{1,1}$ & ----------------- & $\left\langle \delta
_{e_{1},e_{1}}\right\rangle $ & ${\mathcal{E}}_{1,1}$ & $GL\left( {\mathcal{E%
}}_{1,1}\right) $ \\ \hline
\end{tabular}%
$\caption{One-dimensional nilpotent evolution algebras}%
\end{table}%

\subsection{Nilpotent evolution algebras of dimension 2}

Here we consider the classification of $2$-dimensional nilpotent algebras. 

\subsubsection{Nilpotent algebras with annihilator component}

We consider the $1$-dimensional trivial extensions\emph{\ (}corresponding to 
$\theta =0$) of algebras of dimension $1$. We get ${\mathcal{E}}_{2,1}=$ ${%
\mathcal{E}}_{1,1}\oplus $ ${\mathcal{E}}_{1,1}.$

\subsubsection{Nilpotent algebras with no annihilator component}

Here we consider $1$-dimensional extensions of ${\mathcal{E}}_{1,1}$. We get
only one algebra ${\mathcal{E}}_{2,2}:e_{1}^{2}=e_{2}$ corresponding to $%
\theta =\delta _{e_{1},e_{1}}$.

\begin{theorem}
Up to isomorphism there exactly two nilpotent evolution algebras of
dimension $2$ with a natural basis $B=\{e_{1},e_{2}\}$over a field $\mathbb{F%
}$ which are isomorphic to the following pairwise non-isomorphic algebras:%
\begin{table}[H] \centering%
$%
\begin{tabular}{|c|c|c|c|c|c|}
\hline
${\mathcal{E}}$ & Multiplication Table & $\mathcal{B}\left( {\mathcal{%
E\times E}},\mathbb{F}\right) $ & ${\mathcal{H}}\left( {\mathcal{E\times E}},%
\mathbb{F}\right) $ & $ann\left( {\mathcal{E}}\right) $ & $Aut\left( {%
\mathcal{E}}\right) $ \\ \hline
${\mathcal{E}}_{2,1}$ & ----------------- & $0$ & $\left\langle \delta
_{e_{1},e_{1}},\delta _{e_{2},e_{2}}\right\rangle $ & ${\mathcal{E}}_{2,1}$
& $GL\left( {\mathcal{E}}_{2,1}\right) $ \\ \hline
${\mathcal{E}}_{2,2}$ & $e_{1}^{2}=e_{2}.$ & $\left\langle \delta
_{e_{1},e_{1}}\right\rangle $ & $\left\langle \delta
_{e_{2},e_{2}}\right\rangle $ & $\left\langle e_{2}\right\rangle $ & $\phi =%
\begin{bmatrix}
a_{11} & 0 \\ 
a_{21} & a_{11}^{2}%
\end{bmatrix}%
$ \\ \hline
\end{tabular}%
$\caption{Two-dimensional nilpotent evolution algebras}%
\end{table}%
\end{theorem}

\subsection{Nilpotent evolution algebras of dimension 3}

Here we classify three-dimensional nilpotent evolution algebras over
arbitrary fields based on nilpotent evolution algebras of dimension up to
two.

\subsubsection{Nilpotent algebras with annihilator component}

Here we consider the 1-dimensional trivial extensions of algebras of
dimension $2$. We get ${\mathcal{E}}_{3,1}=$ ${\mathcal{E}}_{2,1}\oplus $ ${%
\mathcal{E}}_{1,1}$ and ${\mathcal{E}}_{3,2}=$ ${\mathcal{E}}_{2,2}\oplus $ $%
{\mathcal{E}}_{1,1}.$

\subsubsection{Nilpotent algebras with no annihilator component}

\underline{\textbf{1-dimensional extension of} ${\mathcal{E}}_{2,1}$:}

\bigskip

Let $\theta =\alpha \delta _{e_{1},e_{1}}+\beta \delta _{e_{2},e_{2}}$ and $%
\alpha \beta \neq 0$, otherwise $\theta ^{\bot }\cap ann({\mathcal{E}}%
_{2,1})\neq 0$. The $\alpha ^{-1}\theta $ span the same subspace in ${%
\mathcal{H}}\left( {\mathcal{E\times E}},\mathbb{F}\right) $ as $\theta $,
so we may assume $\alpha =1$. For any $\alpha ,\beta \in $ $\mathbb{F}^{\ast
}$, let $\theta _{\alpha }=\delta _{e_{1},e_{1}}+\alpha \delta
_{e_{2},e_{2}},\theta _{\beta }=\delta _{e_{1},e_{1}}+\beta \delta
_{e_{2},e_{2}}$ and assume that $\left( {\mathcal{E}}_{2,1}\right) _{\theta
_{\alpha }}$ is isomorphic to $\left( {\mathcal{E}}_{2,1}\right) _{\theta
_{\beta }}$ then by Lemma \ref{iso}, there exist a $\phi \in Aut\left( {%
\mathcal{E}}_{2,1}\right) $ and $\lambda \in \mathbb{F}^{\ast }$ such that $%
\phi \theta _{\alpha }=\lambda \theta _{\beta }$. Thus $\phi ^{t}%
\begin{bmatrix}
1 & 0 \\ 
0 & \alpha 
\end{bmatrix}%
\phi =\lambda \allowbreak 
\begin{bmatrix}
1 & 0 \\ 
0 & \beta 
\end{bmatrix}%
\allowbreak $, it follows that $\beta =\left( \delta =\frac{\det \phi }{%
\lambda }\right) ^{2}\alpha $. On the other hand, assume $\beta =$ $\delta
^{2}\alpha $. Then $\phi \theta _{\alpha }=\theta _{\beta }$ where $\phi =%
\begin{bmatrix}
1 & 0 \\ 
0 & \delta 
\end{bmatrix}%
$, then by Lemma \ref{iso} $\left( {\mathcal{E}}_{2,1}\right) _{\theta
_{\alpha }}$ is isomorphic to $\left( {\mathcal{E}}_{2,1}\right) _{\theta
_{\beta }}$. Hence $\left( {\mathcal{E}}_{2,1}\right) _{\theta _{\alpha }}$
is isomorphic to $\left( {\mathcal{E}}_{2,1}\right) _{\theta _{\beta }}$ if
and only if there is a $\delta \in \mathbb{F}^{\ast }$ such that $\beta =$ $%
\delta ^{2}\alpha $. So we get the algebras:

${\mathcal{E}}_{3,3}^{\alpha }:e_{1}^{2}=e_{3},e_{2}^{2}=\alpha e_{3}$ $%
;\alpha \in \mathbb{F}^{\ast }/\mathbb{F}^{\ast 2}.$

\bigskip

\underline{\textbf{1-dimensional extension of }${\mathcal{E}}_{2,2}$:}

\bigskip

We get only one algebra ${\mathcal{E}}_{3,4}:e_{1}^{2}=e_{2},e_{2}^{2}=e_{3}$
corresponding to $\theta =\delta _{e_{2},e_{2}}$.

\bigskip

\underline{\textbf{2-dimensional extension of }${\mathcal{E}}_{1,1}$:}

\bigskip

Here we have $H({\mathcal{E}}_{1,1}\times {\mathcal{E}}_{1,1},\mathbb{F}%
)=\left\langle \delta _{e_{1},e_{1}}\right\rangle $, so there is no
2-dimensional extension of\textbf{\ }${\mathcal{E}}_{1,1}.$

\begin{theorem}
\label{clas3}Up to isomorphism any 3-dimensional nilpotent evolution
algebras with a natural basis $B=\left\{ e_{1},e_{2},e_{3}\right\} $ over a
field $\mathbb{F}$\ is isomorphic to one of the following$\mathcal{\ }$%
pairwise non-isomorphic nilpotent evolution algebras:%
\begin{table}[H] \centering%
\begin{tabular}{|c|c|c|c|c|c|}
\hline
${\mathcal{E}}$ & Multiplication Table & ${\mathcal{B}}\left( {\mathcal{%
E\times E}},\mathbb{F}\right) $ & ${\mathcal{H}}\left( {\mathcal{E\times E}},%
\mathbb{F}\right) $ & $ann\left( {\mathcal{E}}\right) $ & $Aut\left( {%
\mathcal{E}}\right) $ \\ \hline
${\mathcal{E}}_{3,1}$ & \multicolumn{1}{|l|}{-----------------} & $0$ & $%
Sym\left( {\mathcal{E}}_{3,1}\right) $ & ${\mathcal{E}}_{3,1}$ & $GL\left( {%
\mathcal{E}}_{3,1}\right) $ \\ \hline
${\mathcal{E}}_{3,2}$ & \multicolumn{1}{|l|}{$e_{1}^{2}=e_{2}.$} & $%
\left\langle \delta _{e_{1},e_{1}}\right\rangle $ & $\left\langle \delta
_{e_{2},e_{2}},\delta _{e_{3},e_{3}}\right\rangle $ & $\left\langle
e_{2},e_{3}\right\rangle $ & $\phi =%
\begin{bmatrix}
a_{11} & 0 & 0 \\ 
a_{21} & a_{11}^{2} & a_{23} \\ 
a_{31} & 0 & a_{33}%
\end{bmatrix}%
$ \\ \hline
${\mathcal{E}}_{3,3}^{\alpha \in \mathbb{F}^{\ast }\diagup \mathbb{F}^{\ast
2}}$ & \multicolumn{1}{|l|}{$e_{1}^{2}=e_{3},e_{2}^{2}=\alpha e_{3}.$} & $%
\left\langle \delta _{e_{1},e_{1}}+\delta _{e_{2},e_{2}}\right\rangle $ & $%
\left\langle \delta _{e_{1},e_{1}},\delta _{e_{3},e_{3}}\right\rangle $ & $%
\left\langle e_{3}\right\rangle $ & $%
\begin{array}{l}
\phi =%
\begin{bmatrix}
a_{11} & a_{12} & 0 \\ 
a_{21} & a_{22} & 0 \\ 
a_{31} & a_{32} & a_{33}%
\end{bmatrix}%
, \\ 
a_{11}^{2}+\alpha a_{21}^{2}=a_{33}, \\ 
a_{12}^{2}+\alpha a_{22}^{2}=\alpha a_{33}, \\ 
a_{11}a_{12}+\alpha a_{21}a_{22}=0.%
\end{array}%
$ \\ \hline
${\mathcal{E}}_{3,4}$ & \multicolumn{1}{|l|}{$%
e_{1}^{2}=e_{2},e_{2}^{2}=e_{3}.$} & $\left\langle \delta
_{e_{1},e_{1}},\delta _{e_{2},e_{2}}\right\rangle $ & $\left\langle \delta
_{e_{3},e_{3}}\right\rangle $ & $\left\langle e_{3}\right\rangle $ & $\phi =%
\begin{bmatrix}
a_{11} & 0 & 0 \\ 
0 & a_{11}^{2} & 0 \\ 
a_{31} & 0 & a_{11}^{4}%
\end{bmatrix}%
$ \\ \hline
\end{tabular}%
\caption{Three-dimensional nilpotent evolution algebras}%
\end{table}%
\end{theorem}

From Theorem \ref{clas3}, up to isomorphism, the number of 3-dimensional
nilpotent evolution algebras:%
\begin{table}[H] \centering%
\begin{tabular}{|c|c|c|c|}
\hline
Over an algebraic closed field $\mathbb{F\ }$is & Over $%
\mathbb{R}
$ is & Over $%
\mathbb{Q}
$ is & Over any finite field is \\ \hline
4 & 5 & $\infty $ & $%
\begin{tabular}{c}
5$\text{ if characteristic }\mathbb{F}$ is odd \\ 
4$\text{ if characteristic }\mathbb{F}$ is even%
\end{tabular}%
$ \\ \hline
\end{tabular}%
\caption{Number of 3-dimensional nilpotent Evolution algebras}%
\end{table}%

\section{Nilpotent evolution algebras of dimension 4}

\label{dim4}In this section we present a complete classification of $4$%
-dimensional nilpotent evolution algebras over an algebraic closed field and 
$%
\mathbb{R}
$. We denote real nilpotent evolution algebras of dimension $4$ by $E_{4,j}$.

\subsection{Nilpotent algebras with annihilator component}

Here we consider the $1$-dimensional trivial extension\emph{\ }corresponding
to $\theta =0$ of algebras of dimension $3$.\emph{\ }We get ${\mathcal{E}}%
_{4,1}=E_{4,1}=$ ${\mathcal{E}}_{3,1}\oplus $ ${\mathcal{E}}_{1,1},{\mathcal{%
E}}_{4,2}=E_{4,2}={\mathcal{E}}_{3,2}\oplus {\mathcal{E}}_{1,1},{\mathcal{E}}%
_{4,3}=E_{4,3}=$ ${\mathcal{E}}_{3,3}^{\alpha =1}\oplus $ ${\mathcal{E}}%
_{1,1},E_{4,4}=$ ${\mathcal{E}}_{3,3}^{\alpha =-1}\oplus $ ${\mathcal{E}}%
_{1,1}$ and ${\mathcal{E}}_{4,4}=E_{4,5}=$ ${\mathcal{E}}_{3,4}\oplus $ ${%
\mathcal{E}}_{1,1}$.

\subsection{Nilpotent algebras with no annihilator component}

\underline{\textbf{1-dimensional extension of }${\mathcal{E}}_{3,1}:$}

\bigskip

Let $\theta =\alpha \delta _{e_{1},e_{1}}+\beta \delta _{e_{2},e_{2}}+\gamma
\delta _{e_{3},e_{3}}$ and $\alpha \beta \gamma \neq 0$. After dividing we
may assume that $\theta =\delta _{e_{1},e_{1}}+\beta \delta
_{e_{2},e_{2}}+\gamma \delta _{e_{3},e_{3}}$ and $\beta \gamma \neq 0$.
Consider $\phi _{1},\phi _{2},\phi _{3},\phi _{4}\in \mathcal{S}_{\theta
}\left( {\mathcal{E}}_{3,1}\right) $ such that%
\begin{eqnarray*}
\phi _{1} &=&%
\begin{bmatrix}
1 & 0 & 0 \\ 
0 & \beta ^{-\frac{1}{2}} & 0 \\ 
0 & 0 & \gamma ^{-\frac{1}{2}}%
\end{bmatrix}%
,\phi _{2}=%
\begin{bmatrix}
1 & 0 & 0 \\ 
0 & \beta ^{-\frac{1}{2}} & 0 \\ 
0 & 0 & \left( -\gamma \right) ^{-\frac{1}{2}}%
\end{bmatrix}%
,\phi _{3}=%
\begin{bmatrix}
1 & 0 & 0 \\ 
0 & 0 & \left( -\beta \right) ^{-\frac{1}{2}} \\ 
0 & \gamma ^{-\frac{1}{2}} & 0%
\end{bmatrix}%
, \\
\phi _{4} &=&%
\begin{bmatrix}
0 & 0 & 1 \\ 
0 & \left( -\beta \right) ^{-\frac{1}{2}} & 0 \\ 
\left( -\gamma \right) ^{-\frac{1}{2}} & 0 & 0%
\end{bmatrix}%
.
\end{eqnarray*}

\begin{itemize}
\item \textit{Over an algebraic closed field} $\mathbb{F}$:

Here $\beta ,\gamma $ are squares then $\phi _{1}\theta =\delta
_{e_{1},e_{1}}+\delta _{e_{2},e_{2}}+\delta _{e_{3},e_{3}}$. So we get the
algebra ${\mathcal{E}}_{4,5}:e_{1}^{2}=e_{2}^{2}=e_{3}^{2}=e_{4}$.

\item \textit{Over the real field }$%
\mathbb{R}
$\textit{:}

\begin{itemize}
\item If $\beta >0,\gamma >0$ then $\phi _{1}\theta =\theta _{1}=\delta
_{e_{1},e_{1}}+\delta _{e_{2},e_{2}}+\delta _{e_{3},e_{3}}$.

\item If $\beta >0,\gamma <0$ then $\phi _{2}\theta =\theta _{2}=\delta
_{e_{1},e_{1}}+\delta _{e_{2},e_{2}}-\delta _{e_{3},e_{3}}$.

\item If $\beta <0,\gamma >0$ then $\phi _{3}\theta =\theta _{2}=\delta
_{e_{1},e_{1}}+\delta _{e_{2},e_{2}}-\delta _{e_{3},e_{3}}$.

\item If $\beta <0,\gamma <0$ then $\phi _{4}\theta =-\theta _{2}=-\left(
\delta _{e_{1},e_{1}}+\delta _{e_{2},e_{2}}-\delta _{e_{3},e_{3}}\right) $.
\end{itemize}

Therefore we have two representatives $\theta _{1},\theta _{2}$ .\ It
remains to decide isomorphism between $\left( {\mathcal{E}}_{3,1}\right)
_{\theta _{1}},\left( {\mathcal{E}}_{3,1}\right) _{\theta _{2}}$. By Lemma %
\ref{iso}, $\left( {\mathcal{E}}_{3,1}\right) _{\theta _{1}}$is isomorphic
to $\left( {\mathcal{E}}_{3,1}\right) _{\theta _{2}}$ if and only if there
exist a $\phi \in Aut\left( {\mathcal{E}}_{3,1}\right) $ and $\lambda \in 
\mathbb{R}
^{\ast }$ such that $\phi \theta _{1}=\lambda \theta _{2}$. This amounts to
(among others) the equations:%
\begin{equation*}
a_{11}^{2}+a_{21}^{2}+a_{31}^{2}=\lambda
,a_{12}^{2}+a_{22}^{2}+a_{32}^{2}=\lambda
,a_{13}^{2}+a_{23}^{2}+a_{33}^{2}=-\lambda ,
\end{equation*}%
which never have a solution in $%
\mathbb{R}
$. Thus $\left( {\mathcal{E}}_{3,1}\right) _{\theta _{1}},\left( {\mathcal{E}%
}_{3,1}\right) _{\theta _{2}}$ are non-isomorphic. So we get only two
algebras:

$E_{4,6}:e_{1}^{2}=e_{2}^{2}=e_{3}^{2}=e_{4}.$

$E_{4,7}:e_{1}^{2}=e_{2}^{2}=e_{4},e_{3}^{2}=-e_{4}.$
\end{itemize}

\underline{\textbf{1-dimensional extension of }${\mathcal{E}}_{3,2}:$}

\bigskip

Let $\theta =\alpha \delta _{e_{2},e_{2}}+\beta \delta _{e_{3},e_{3}}$ and $%
\alpha \beta \neq 0$. We can divide to get $\alpha =1$.\ So we may assume
that $\theta =\delta _{e_{2},e_{2}}+\beta \delta _{e_{3},e_{3}}$. Consider $%
\phi _{1},\phi _{2}\in \mathcal{S}_{\theta }\left( {\mathcal{E}}%
_{3,2}\right) $ such that 
\begin{equation*}
\phi _{1}=%
\begin{bmatrix}
1 & 0 & 0 \\ 
0 & 1 & 0 \\ 
0 & 0 & \beta ^{-\frac{1}{2}}%
\end{bmatrix}%
,\phi _{2}=%
\begin{bmatrix}
1 & 0 & 0 \\ 
0 & 1 & 0 \\ 
0 & 0 & \left( -\beta \right) ^{-\frac{1}{2}}%
\end{bmatrix}%
.
\end{equation*}

\begin{itemize}
\item \textit{Over an algebraic closed field }$\mathbb{F}$\textit{:}

Here $\beta $ is a square then $\phi _{1}\theta =\delta
_{e_{2},e_{2}}+\delta _{e_{3},e_{3}}$. So we get the algebra ${\mathcal{E}}%
_{4,6}:e_{1}^{2}=e_{2},e_{2}^{2}=e_{3}^{2}=e_{4}$.

\item \textit{Over the real field }$%
\mathbb{R}
$\textit{:}

\begin{itemize}
\item If $\beta >0$ then $\phi _{1}\theta =\theta _{1}=\delta
_{e_{2},e_{2}}+\delta _{e_{3},e_{3}}$.

\item If $\beta <0$ then $\phi _{2}\theta =\theta _{2}=\delta
_{e_{2},e_{2}}-\delta _{e_{3},e_{3}}$.
\end{itemize}

Suppose that $\phi \in Aut\left( {\mathcal{E}}_{3,2}\right) $ and $\lambda
\in 
\mathbb{R}
^{\ast }$ such that $\phi \theta _{1}=\lambda \theta _{2}$. Then we get
(among others) the equations 
\begin{equation*}
a_{22}^{2}+a_{32}^{2}=\lambda ,a_{33}^{2}=-\lambda ,
\end{equation*}%
which never have a solutions in $%
\mathbb{R}
$. Thus $\left( {\mathcal{E}}_{3,2}\right) _{\theta _{1}},\left( {\mathcal{E}%
}_{3,2}\right) _{\theta _{2}}$ are not isomorphic. So we get only two
algebras:

$E_{4,8}:e_{1}^{2}=e_{2},e_{2}^{2}=e_{3}^{2}=e_{4}.$

$E_{4,9}:e_{1}^{2}=e_{2},e_{2}^{2}=e_{4},e_{3}^{2}=-e_{4}$.
\end{itemize}

\underline{\textbf{1-dimensional extension of }${\mathcal{E}}_{3,3}^{\alpha
} $:}

\bigskip

Let $\theta =\gamma \delta _{e_{1},e_{1}}+\beta \delta _{e_{3},e_{3}}$ and $%
\beta \neq 0$. After dividing we may assume that $\theta =\gamma \left(
\delta _{e_{1},e_{1}}+\alpha \delta _{e_{2},e_{2}}\right) +\delta
_{e_{3},e_{3}}$. Consider $\phi _{1},\phi _{2},\phi _{3}\in \mathcal{S}%
_{\theta }\left( {\mathcal{E}}_{3,3}^{\alpha }\right) $ such that 
\begin{equation*}
\phi _{1}=%
\begin{bmatrix}
\left( -\gamma \right) ^{\frac{1}{2}} & 0 & 0 \\ 
0 & \left( -\gamma \right) ^{\frac{1}{2}} & 0 \\ 
0 & 0 & -\gamma 
\end{bmatrix}%
,\phi _{2}=%
\begin{bmatrix}
\gamma ^{\frac{1}{2}} & 0 & 0 \\ 
0 & \gamma ^{\frac{1}{2}} & 0 \\ 
0 & 0 & \gamma 
\end{bmatrix}%
,\phi _{3}=%
\begin{bmatrix}
0 & \left( -\gamma \right) ^{\frac{1}{2}} & 0 \\ 
\left( -\gamma \right) ^{\frac{1}{2}} & 0 & 0 \\ 
0 & 0 & -\gamma 
\end{bmatrix}%
.
\end{equation*}

\begin{itemize}
\item \textit{Over an algebraic closed field }$\mathbb{F}$\textit{. }In this
case $\alpha =1$.

If $\gamma =0$ then we get $\theta _{1}=\delta _{e_{3},e_{3}}$. If $\gamma
\neq 0$ then $\phi _{2}\theta =\gamma ^{2}\left( \delta
_{e_{1},e_{1}}+\delta _{e_{3},e_{3}}\right) $. Then after dividing we may
assume that $\theta _{2}=\delta _{e_{1},e_{1}}+\delta _{e_{3},e_{3}}$.
Suppose that $\phi \in Aut\left( {\mathcal{E}}_{3,3}^{\alpha =1}\right) $
and $\lambda \in \mathbb{F}^{\ast }$ such that $\phi \theta _{1}=\lambda
\theta _{2}$. 
\begin{equation*}
a_{31}^{2}-a_{32}^{2}=\lambda ,a_{33}^{2}=\lambda
,a_{31}a_{32}=a_{31}a_{33}=a_{32}a_{33}=0,
\end{equation*}%
which never have a solution.  So, $\left( {\mathcal{E}}_{3,3}^{\alpha
=1}\right) _{\theta _{1}}$ is not isomorphic to $\left( {\mathcal{E}}%
_{3,3}^{\alpha =1}\right) _{\theta _{2}}$.Therefore we get the following
non-isomorphic algebras:

${\mathcal{E}}_{4,7}:e_{1}^{2}=e_{3},e_{2}^{2}=e_{3},e_{3}^{2}=e_{4}$.

${\mathcal{E}}_{4,8}:e_{1}^{2}=e_{3}+e_{4},e_{2}^{2}=e_{3},e_{3}^{2}=e_{4}$.

\item \textit{Over the real field }$%
\mathbb{R}
$\textit{. In this case }$\alpha =\pm 1$.

\underline{\textrm{Case1.} $\alpha =1$:}

\begin{itemize}
\item If $\gamma =0$, then we get $\theta _{1}=\delta _{e_{3},e_{3}}$.

\item If $\gamma <0$ then $\phi _{3}\theta =\gamma ^{2}\left( -\delta
_{e_{2},e_{2}}+\delta _{e_{3},e_{3}}\right) $. Since ${\mathcal{B}}\left( {%
\mathcal{E}}_{3,3}^{\alpha =1}{\mathcal{\times E}}_{3,3}^{\alpha =1},%
\mathbb{R}
\right) =\left\langle \delta _{e_{1},e_{1}}+\delta
_{e_{2},e_{2}}\right\rangle $, $\delta _{e_{1},e_{1}}=-\delta _{e_{2},e_{2}}$
in ${\mathcal{H}}\left( {\mathcal{E}}_{3,3}^{\alpha =1}{\mathcal{\times E}}%
_{3,3}^{\alpha =1},%
\mathbb{R}
\right) $. So, $\phi _{3}\theta =\gamma ^{2}\left( \delta
_{e_{1},e_{1}}+\delta _{e_{3},e_{3}}\right) $. So after dividing we may
assume that $\theta _{2}=\delta _{e_{1},e_{1}}+\delta _{e_{3},e_{3}}$.

\item If $\gamma >0$ then $\phi _{2}\theta =\gamma ^{2}\left( \delta
_{e_{1},e_{1}}+\delta _{e_{3},e_{3}}\right) $. Therefore we a gain get $%
\theta _{2}=\delta _{e_{1},e_{1}}+\delta _{e_{3},e_{3}}$.
\end{itemize}

Suppose that $\phi \in Aut\left( {\mathcal{E}}_{3,3}^{\alpha =1}\right) $
and $\lambda \in \mathbb{F}^{\ast }$ such that $\phi \theta _{1}=\lambda
\theta _{2}$. This then amounts to the following equations: 
\begin{equation*}
a_{31}^{2}-a_{32}^{2}=\lambda ,a_{33}^{2}=\lambda
,a_{31}a_{32}=a_{31}a_{33}=a_{32}a_{33}=0,
\end{equation*}%
which never have a solution in $%
\mathbb{R}
$. So, $\left( {\mathcal{E}}_{3,3}^{\alpha =1}\right) _{\theta _{1}}$ is not
isomorphic to $\left( {\mathcal{E}}_{3,3}^{\alpha =1}\right) _{\theta _{2}}$%
.Therefore we get the following non-isomorphic algebras:

$E_{4,10}:e_{1}^{2}=e_{3},e_{2}^{2}=e_{3},e_{3}^{2}=e_{4}.$

$E_{4,11}:e_{1}^{2}=e_{3}+e_{4},e_{2}^{2}=e_{3},e_{3}^{2}=e_{4}.$

\underline{\textrm{Case 2.} $\alpha =-1$:}

\begin{itemize}
\item If $\gamma =0$ then we get $\theta _{1}=\delta _{e_{3},e_{3}}$.

\item If $\gamma <0$ then $\phi _{3}\theta =-\gamma ^{2}\left( \delta
_{e_{1},e_{1}}-\delta _{e_{3},e_{3}}\right) $. So after dividing we may
assume that $\theta _{2}=\delta _{e_{1},e_{1}}-\delta _{e_{3},e_{3}}$.

\item If $\gamma >0$ then $\phi _{2}\theta =\gamma ^{2}\left( \delta
_{e_{1},e_{1}}+\delta _{e_{3},e_{3}}\right) $. So after dividing we may
assume that $\theta _{3}=\delta _{e_{1},e_{1}}+\delta _{e_{3},e_{3}}$.
\end{itemize}

Firstly, suppose that $\phi \in Aut\left( {\mathcal{E}}_{3,3}^{\alpha
=-1}\right) $ and $\lambda \in 
\mathbb{R}
^{\ast }$ such that $\phi \theta _{1}=\lambda \theta _{2}$. Then we get the
following equations 
\begin{equation*}
a_{31}^{2}+a_{32}^{2}=\lambda ,a_{33}^{2}=-\lambda
,a_{31}a_{32}=a_{31}a_{33}=a_{32}a_{33}=0,
\end{equation*}%
which never have a solution in $%
\mathbb{R}
$. So $\left( {\mathcal{E}}_{3,3}^{\alpha =-1}\right) _{\theta _{1}}$ is not
isomorphic to $\left( {\mathcal{E}}_{3,3}^{\alpha =-1}\right) _{\theta _{2}}$%
.

Next, suppose that $\phi \in Aut\left( {\mathcal{E}}_{3,3}^{\alpha
=-1}\right) $ and $\lambda \in 
\mathbb{R}
^{\ast }$ such that $\phi \theta _{1}=\lambda \theta _{3}$. Then we get the
equations 
\begin{equation*}
a_{31}^{2}+a_{32}^{2}=\lambda ,a_{33}^{2}=\lambda
,a_{31}a_{32}=a_{31}a_{33}=a_{32}a_{33}=0,
\end{equation*}%
which never have a solution in $%
\mathbb{R}
$. So $\left( {\mathcal{E}}_{3,3}^{\alpha =-1}\right) _{\theta _{1}}$ is not
isomorphic to $\left( {\mathcal{E}}_{3,3}^{\alpha =-1}\right) _{\theta _{3}}$%
.

Finally, suppose that $\phi \in Aut\left( {\mathcal{E}}_{3,3}^{\alpha
=-1}\right) $ and $\lambda \in 
\mathbb{R}
^{\ast }$ such that $\phi \theta _{3}=\lambda \theta _{2}$. Then we get
(among others) the equations%
\begin{equation*}
a_{11}^{2}+a_{31}^{2}+a_{12}^{2}+a_{32}^{2}=\lambda ,a_{33}^{2}=-\lambda ,
\end{equation*}%
which never have a solution in $%
\mathbb{R}
$. So $\left( {\mathcal{E}}_{3,3}^{\alpha =-1}\right) _{\theta _{2}}$ is not
isomorphic to $\left( {\mathcal{E}}_{3,3}^{\alpha =-1}\right) _{\theta _{3}}$%
. Hence we get the following pairwise non-isomorphic algebras:

$E_{4,12}:e_{1}^{2}=e_{3},e_{2}^{2}=-e_{3},e_{3}^{2}=e_{4}.$

$E_{4,13}:e_{1}^{2}=e_{3}+e_{4},e_{2}^{2}=-e_{3},e_{3}^{2}=-e_{4}.$

$E_{4,14}:e_{1}^{2}=e_{3}+e_{4},e_{2}^{2}=-e_{3},e_{3}^{2}=e_{4}.$
\end{itemize}

\underline{\textbf{1-dimensional extension of }${\mathcal{E}}_{3,4}:$}

\begin{itemize}
\item \textit{Over an algebraic closed field }$\mathbb{F}$:

Since ${\mathcal{H}}\left( {\mathcal{E}}_{3,4}{\mathcal{\times E}}_{3,4},%
\mathbb{F}\right) =\left\langle \delta _{e_{3},e_{3}}\right\rangle $, we get
only one algebra ${\mathcal{E}}%
_{4,9}:e_{1}^{2}=e_{2},e_{2}^{2}=e_{3},e_{3}^{2}=e_{4}$.

\item \textit{Over the real field }$%
\mathbb{R}
$\textit{:}

Also, ${\mathcal{H}}\left( {\mathcal{E}}_{3,4}{\mathcal{\times E}}_{3,4},%
\mathbb{R}
\right) =\left\langle \delta _{e_{3},e_{3}}\right\rangle $. So, we get only
one algebra $E_{4,15}:e_{1}^{2}=e_{2},e_{2}^{2}=e_{3},e_{3}^{2}=e_{4}$.
\end{itemize}

\underline{\textbf{2-dimensional extension of }${\mathcal{E}}_{2,1}$:}

\begin{itemize}
\item \textit{Over an algebraic closed field }$\mathbb{F}$:

As ${\mathcal{H}}\left( {\mathcal{E}}_{2,1}{\mathcal{\times E}}_{2,1},%
\mathbb{F}\right) =\left\langle \delta _{e_{1},e_{1}},\delta
_{e_{2},e_{2}}\right\rangle $, we get only one algebra ${\mathcal{E}}%
_{4,10}:e_{1}^{2}=e_{3},e_{2}^{2}=e_{4}$.

\item \textit{Over the real field }$%
\mathbb{R}
$\textit{:}

Also, ${\mathcal{H}}\left( {\mathcal{E}}_{2,1}{\mathcal{\times E}}_{2,1},%
\mathbb{R}
\right) =\left\langle \delta _{e_{1},e_{1}},\delta
_{e_{2},e_{2}}\right\rangle $ and so we get only one algebra $%
E_{4,16}:e_{1}^{2}=e_{3},e_{2}^{2}=e_{4}$.
\end{itemize}

\underline{\textbf{2-dimensional extension of }${\mathcal{E}}_{2,2}$:}

\bigskip

Here ${\mathcal{H}}\left( {\mathcal{E}}_{2,2}\times {\mathcal{E}}_{2,2},%
\mathbb{F}\right) =\left\langle \delta _{e_{2},e_{2}}\right\rangle $ and
hence ${\mathcal{E}}_{2,2}$ has no 2-dimensional extension.

\bigskip

\underline{\textbf{3-dimensional extension of }${\mathcal{E}}_{1,1}$:}

\bigskip

Here ${\mathcal{H}}\left( {\mathcal{E}}_{1,1}\times {\mathcal{E}}_{1,1},%
\mathbb{F}\right) =\left\langle \delta _{e_{1},e_{1}}\right\rangle $ and
hence ${\mathcal{E}}_{1,1}$ has no 3-dimensional extension.

\begin{theorem}
Up to isomorphism there exist $10$ nilpotent evolution algebras of dimension
four with a natural basis $B=\left\{ e_{1},e_{2},e_{3},e_{4}\right\} $ over
an algebraic closed field $\mathbb{F}$ which are isomorphic to one of the
following\ pairwise non-isomorphic nilpotent evolution algebras:

\begin{itemize}
\item ${\mathcal{E}}_{4,1}:\mathcal{\ }$All products are zeros.

\item ${\mathcal{E}}_{4,2}:e_{1}^{2}=e_{2}.$

\item ${\mathcal{E}}_{4,3}:e_{1}^{2}=e_{3},e_{2}^{2}=e_{3}.$

\item ${\mathcal{E}}_{4,4}:e_{1}^{2}=e_{2},e_{2}^{2}=e_{3}.$

\item ${\mathcal{E}}_{4,5}:e_{1}^{2}=e_{4},e_{2}^{2}=e_{4},e_{3}^{2}=e_{4}.$

\item ${\mathcal{E}}_{4,6}:e_{1}^{2}=e_{2},e_{2}^{2}=e_{4},e_{3}^{2}=e_{4}.$

\item ${\mathcal{E}}_{4,7}:e_{1}^{2}=e_{3},e_{2}^{2}=e_{3},e_{3}^{2}=e_{4}.$

\item ${\mathcal{E}}%
_{4,8}:e_{1}^{2}=e_{3}+e_{4},e_{2}^{2}=e_{3},e_{3}^{2}=e_{4}.$

\item ${\mathcal{E}}_{4,9}:e_{1}^{2}=e_{2},e_{2}^{2}=e_{3},e_{3}^{2}=e_{4}.$

\item ${\mathcal{E}}_{4,10}:e_{1}^{2}=e_{3},e_{2}^{2}=e_{4}.$
\end{itemize}
\end{theorem}

\begin{theorem}
Up to isomorphism there exist $16$ nilpotent evolution algebras of dimension
four with a natural basis $B=\left\{ e_{1},e_{2},e_{3},e_{4}\right\} $ over $%
\mathbb{R}
$ which are isomorphic to one of the following$\mathcal{\ }$pairwise
non-isomorphic nilpotent evolution algebras:

\begin{itemize}
\item $E_{4,1}:\mathcal{\ }$All products are zeros.

\item $E_{4,2}:e_{1}^{2}=e_{2}.$

\item $E_{4,3}:e_{1}^{2}=e_{3},e_{2}^{2}=e_{3}.$

\item $E_{4,4}:e_{1}^{2}=e_{3},e_{2}^{2}=-e_{3}.$

\item $E_{4,5}:e_{1}^{2}=e_{2},e_{2}^{2}=e_{3}.$

\item $E_{4,6}:e_{1}^{2}=e_{4},e_{2}^{2}=e_{4},e_{3}^{2}=e_{4}.$

\item $E_{4,7}:e_{1}^{2}=e_{4},e_{2}^{2}=e_{4},e_{3}^{2}=-e_{4}.$

\item $E_{4,8}:e_{1}^{2}=e_{2},e_{2}^{2}=e_{4},e_{3}^{2}=e_{4}.$

\item $E_{4,9}:e_{1}^{2}=e_{2},e_{2}^{2}=e_{4},e_{3}^{2}=-e_{4}.$

\item $E_{4,10}:e_{1}^{2}=e_{3},e_{2}^{2}=e_{3},e_{3}^{2}=e_{4}.$

\item $E_{4,11}:e_{1}^{2}=e_{3}+e_{4},e_{2}^{2}=e_{3},e_{3}^{2}=e_{4}.$

\item $E_{4,12}:e_{1}^{2}=e_{3},e_{2}^{2}=-e_{3},e_{3}^{2}=e_{4}.$

\item $E_{4,13}:e_{1}^{2}=e_{3}+e_{4},e_{2}^{2}=-e_{3},e_{3}^{2}=-e_{4}.$

\item $E_{4,14}:e_{1}^{2}=e_{3}+e_{4},e_{2}^{2}=-e_{3},e_{3}^{2}=e_{4}.$

\item $E_{4,15}:e_{1}^{2}=e_{2},e_{2}^{2}=e_{3},e_{3}^{2}=e_{4}$

\item $E_{4,16}:e_{1}^{2}=e_{3},e_{2}^{2}=e_{4}.$
\end{itemize}
\end{theorem}

\end{document}